\documentclass[11pt]{amsart}

\usepackage{amscd,amssymb,amsmath,graphicx,verbatim}
\usepackage{wasysym}
\usepackage{mathrsfs}
\usepackage{hyperref}
\usepackage[TS1,OT1,T1]{fontenc}
\usepackage{lscape} 
\usepackage{fullpage} 
\usepackage{pslatex} 
\usepackage{tikz}
\usetikzlibrary{shapes,arrows}
\usepackage[all]{xy}
\newtheorem{theorem}{Theorem}[section]
\newtheorem{lemma}[theorem]{Lemma}
\newtheorem{corollary}[theorem]{Corollary}
\newtheorem{proposition}[theorem]{Proposition}

\theoremstyle{definition}
\newtheorem{definition}[theorem]{Definition}

\newtheorem{example}[theorem]{Example}

\theoremstyle{remark}

 \DeclareMathOperator{\Hom}{Hom}

\DeclareMathOperator{\Ext}{Ext}

\DeclareMathOperator*{\Modr}{\mathsf{Mod}-\!}

\newcommand{\Tcal}{\ensuremath{\mathcal{T}}}
\newcommand{\T}{\ensuremath{\mathcal{T}}}
\newcommand{\Gcal}{\ensuremath{\mathcal{G}}}

\newcommand{\Rcal}{\ensuremath{\mathcal{R}}}

\newcommand{\Bcal}{\ensuremath{\mathcal{B}}}

\newcommand{\Acal}{\ensuremath{\mathcal{A}}}
\newcommand{\Wcal}{\ensuremath{\mathcal{W}}}
\newcommand{\Hcal}{\ensuremath{\mathcal{H}}}

\newcommand{\Ucal}{\ensuremath{\mathcal{U}}}
\newcommand{\Vcal}{\ensuremath{\mathcal{V}}}

\newcommand{\Qcal}{\ensuremath{\mathcal{Q}}}

\newcommand{\Add}{\mathsf{Add}}
\newcommand{\Prod}{\mathsf{Prod}}

\DeclareMathOperator{\Ker}{Ker}

\numberwithin{equation}{section}

\begin{document}

\title{Torsion pairs in silting theory}
\author{Lidia Angeleri H\"ugel, Frederik Marks, Jorge Vit{\'o}ria}
\address{Lidia Angeleri H\"ugel, Dipartimento di Informatica - Settore di Matematica, Universit\`a degli Studi di Verona, Strada le Grazie 15 - Ca' Vignal, I-37134 Verona, Italy} \email{lidia.angeleri@univr.it}
\address{Frederik Marks, Institut f\"ur Algebra und Zahlentheorie, Universit\"at Stuttgart, Pfaffenwaldring 57, 70569 Stuttgart, Germany}
\email{marks@mathematik.uni-stuttgart.de}
\address{Jorge Vit\'oria, Department of Mathematics, City, University of London, Northampton Square, London EC1V 0HB, United Kingdom}
\email{jorge.vitoria@city.ac.uk}
\keywords{torsion pair, silting, cosilting, t-structure, Grothendieck category}
\subjclass[2010]{18E15,18E30,18E40}
\maketitle


\begin{abstract}
In the setting of compactly generated triangulated categories, we show that the heart of a (co)silting t-structure is a Grothendieck category if and only if the (co)silting object satisfies a purity assumption. Moreover, in the cosilting case the previous conditions are  related to the  coaisle of the t-structure being a definable subcategory. If we further assume our triangulated category to be algebraic, it follows that the heart of any nondegenerate compactly generated t-structure is a Grothendieck category.
\end{abstract}


\section{Introduction}
Silting and cosilting objects in triangulated categories are useful generalisations of tilting and cotilting objects. While (co)tilting objects have been a source of many interactions with torsion and localisation theory, it is in the setting of (co)silting objects that classification results occur more naturally. This paper strengthens this claim by showing that, in the setting of compactly generated triangulated categories, relevant torsion-theoretic structures are parametrised by suitable classes of (co)silting objects.

The concept of a silting object,  first introduced in \cite{KV} in the context of derived module categories over finite dimensional hereditary algebras, has recently been extended to  the setting of abstract triangulated categories (\cite{AI,MSSS,NSZ,PV}). 
In this paper, our focus is on t-structures and co-t-structures arising from (co)silting objects. For this purpose, we use the vast theory of \textit{purity} in compactly generated triangulated categories, where a central role is played by the category of contravariant functors on the compact objects. We show that a fundamental property of  the t-structure associated to a cosilting object $C$ -  namely, its heart  being a Grothendieck abelian category - is related to the pure-injectivity of $C$.  An analogous result  holds true for silting objects. Moreover, it turns out that in the cosilting case the pure-injectivity of $C$ is further related to the definability (in terms of coherent functors) of the coaisle of the associated t-structure. We can summarise our results as follows.\\

\textbf{Theorem} (Theorems \ref{pure-inj Grothendieck} and \ref{triple is cosilting}, Corollary \ref{comp gen t-str})
\textit{Let $(\Ucal,\Vcal,\Wcal)$ be a triple in a compactly generated triangulated category $\T$ such that $(\Ucal,\Vcal)$ is a nondegenerate t-structure and $(\Vcal,\Wcal)$ is a co-t-structure. Then the following are equivalent.
\begin{enumerate}
\item $\Vcal$ is definable in $\T$;
\item $\Vcal={}^{\perp_{>0}}C$ for a pure-injective cosilting object $C$ in $\T$;
\item $\Hcal:=\Ucal[-1]\cap\Vcal$ is a Grothendieck category.
\end{enumerate}
In particular, if we further assume $\T$ to be algebraic, it follows that any nondegenerate compactly generated t-structure in $\T$ has a Grothendieck heart.\\ }

For partial results in this direction we refer to  \cite[Proposition 4.2]{NSZ} and   \cite[Corollary 2.5]{BP}. In a forthcoming paper (\cite{MV2}), it will be proved that cosilting complexes in derived module categories are always pure-injective and give rise to definable subcategories as above. {We do not know, however, if the same holds true for arbitrary cosilting objects in compactly generated triangulated categories.} Moreover, it will be shown {in \cite{MV2}} that there are cosilting complexes (in fact, cosilting modules) inducing triples $(\Ucal,\Vcal,\Wcal)$ as above such that the t-structure has a Grothendieck heart, although it is not compactly generated. This will answer \cite[Question 3.5]{BP}.

\smallskip

The structure of the paper is as follows. In Section 2, we present our setup and provide the reader with some preliminaries on torsion pairs and (co)silting objects. In Section 3, we briefly recall the key concepts of pure-projectivity and pure-injectivity and we establish the connection between (co)silting objects having such properties and t-structures with  Grothendieck hearts. Finally, in Section 4, we discuss definable subcategories and we prove the above mentioned relation between pure-injective cosilting objects and certain definable subcategories of the underlying triangulated category.

\medskip
\noindent {\bf Acknowledgements.}  
The authors would like to thank Mike Prest for discussions leading to Corollary \ref{definable subcat}. The second named author was supported by a grant within the DAAD P.R.I.M.E. program. The third named author acknowledges support from the Department of Computer Sciences of the University of Verona in the earlier part of this project, as well as from the Engineering and Physical Sciences Research Council of the United Kingdom, grant number EP/N016505/1, in the later part of this project. {Finally, the authors acknowledge funding from the Project ``Ricerca di Base 2015'' of the University of Verona.}


\section{Preliminaries}
\subsection{Setup and notation} Throughout, we denote by $\T$ a \textbf{compactly generated} triangulated category, i.e. a triangulated category with coproducts for which the subcategory of compact objects, denoted by $\T^c$, has only a set of isomorphism classes and such that for any $Y$ in $\T$ with $\Hom_\T(X,Y)=0$ for all $X$ in $\T^c$, we have $Y=0$. Since $\T$ admits arbitrary set-indexed coproducts, it is idempotent complete (see \cite[Proposition 1.6.8]{Neeman}). It is also well-known (see \cite[Proposition 8.4.6 and Theorem 8.3.3]{Neeman}) that such triangulated categories admit products. In some places, we will further assume $\T$ to be \textbf{algebraic}, i.e. $\T$ can be constructed as the stable category of a Frobenius exact category (see \cite{H}). Note that algebraic and compactly generated triangulated categories are essentially derived categories of small differential graded categories (\cite{Keller}).

All subcategories considered are strict and full. For a set of integers $I$ (which is often expressed by symbols such as $>n$, $<n$, $\geq n$, $\leq n$, $\neq n$, or just $n$, with the obvious associated meaning) we define the following orthogonal classes
$${}^{\perp_I}X:=\{Y\in \T:\Hom_\T(Y,X[i])=0, \forall i\in I\}\ \ \ \ {X}^{\perp_I}:=\{Y\in \T:\Hom_\T(X,Y[i])=0, \forall i\in I\}.$$
If $\mathscr{C}$ is a subcategory of $\Tcal$, then we denote by $\Add(\mathscr{C})$ (respectively, $\Prod(\mathscr{C})$) the smallest subcategory of $\T$ containing $\mathscr{C}$ and closed under coproducts (respectively, products) and summands. If $\mathscr{C}$ consists of a single object $M$, we write $\Add(M)$ and $\Prod(M)$ for the respective subcategories. For a ring $A$, we denote by $\mathsf{Mod}(A)$ the category of right $A$-modules and by $\mathsf{D}(A)$ the unbounded derived category of $\mathsf{Mod}(A)$. The subcategories of injective and of projective $A$-modules are denoted, respectively, by $\mathsf{Inj}(A)$ and $\mathsf{Proj}(A)$, and their bounded homotopy categories by $\mathsf{K}^b(\mathsf{Inj(A)})$ and $\mathsf{K}^b(\mathsf{Proj}(A))$, respectively.


\subsection{Torsion pairs}
We consider the notion of a torsion pair in a triangulated category (see, for example, \cite{IY}), which gives rise to the notions of a t-structure (\cite{BBD}) and a co-t-structure (\cite{Bondarko}, \cite{Pauk}).
\begin{definition}\label{def tp}
A pair of subcategories $(\Ucal,\Vcal)$ in $\T$ is said to be a \textbf{torsion pair} if
\begin{enumerate}
\item $\Ucal$ and $\Vcal$ are closed under summands;
\item $\Hom_\T(\Ucal,\Vcal)=0$;
\item For every object $X$ of $\T$, there are $U$ in $\Ucal$, $V$ in $\Vcal$ and a triangle
$$U\longrightarrow X\longrightarrow V\longrightarrow U[1].$$
\end{enumerate}
In a torsion pair $(\Ucal,\Vcal)$, the class $\Ucal$ is called the \textbf{aisle}, the class $\Vcal$ the \textbf{coaisle}, and $(\Ucal,\Vcal)$ is said to be
\begin{itemize}
\item $\textbf{nondegenerate}$ if $\cap_{n\in\mathbb{Z}}\Ucal[n]=0=\cap_{n\in\mathbb{Z}}\Vcal[n]$;
\item a \textbf{t-structure} if $\Ucal[1]\subseteq \Ucal$, in which case we say that $\Ucal[-1]\cap \Vcal$ is the \textbf{heart} of $(\Ucal,\Vcal)$;
\item a \textbf{co-t-structure} if $\Ucal[-1]\subseteq \Ucal$, in which case we say that $\Ucal\cap \Vcal[-1]$ is the \textbf{coheart} of $(\Ucal,\Vcal)$.
\end{itemize}
\end{definition}

It follows from \cite{BBD} that the heart $\Hcal_\mathbb{T}$ of a t-structure $\mathbb{T}:=(\Ucal,\Vcal)$ in $\T$ is an abelian category with the exact structure induced by the triangles of $\T$ lying in $\Hcal_\mathbb{T}$.
Moreover, the triangle in Definition \ref{def tp}(3) can  be expressed functorially as 
$$\xymatrix{u(X)\ar[r]^{\ f} & X\ar[r]^{g \ \ \ }& v(X)\ar[r]& u(X)[1]}$$
where  $u:\T\longrightarrow \Ucal$ is the right adjoint of the inclusion of $\Ucal$ in $\Tcal$ and $v:\T\longrightarrow \Vcal$ is the left adjoint of the inclusion of $\Vcal$ in $\Tcal$. The existence of one of these adjoints, usually called \textbf{truncation functors}, is in fact equivalent to the fact that $(\Ucal,\Vcal)$ is a t-structure (\cite[Proposition 1.1]{KV}). Observe that the maps $f$ and $g$ in the triangle are, respectively, the counit and unit map of the relevant adjunction.
In particular, it follows that if $f=0$ (respectively, $g=0$), then $u(X)=0$ (respectively, $v(X)=0$). 
{Furthermore, $u$ and $v$ give rise to a cohomological functor defined by:}
\[
H^0_\mathbb{T}\colon \T\longrightarrow \mathcal{H}_\mathbb{T}, \ X \ \mapsto \textsf{H}^0_{\mathbb{T}}(X):=v(u(X[1])[-1])=u(v(X)[1])[-1].
\]
 {Recall that an additive covariant functor from $\Tcal$ to an abelian category $\Acal$ is said to be \textbf{cohomological} if it sends triangles in $\T$ to long exact sequences in $\Acal$.}
 
\medskip

We will also be interested in the properties of torsion pairs generated or cogenerated by certain subcategories of $\T$, which are defined as follows.

\begin{definition}
Let $(\Ucal,\Vcal)$ be a torsion pair in $\T$ and $\Acal$ a subcategory of $\T$. We say that $(\Ucal,\Vcal)$ is
\begin{itemize} 
\item \textbf{generated by $\Acal$} if $(\Ucal,\Vcal)=({}^{\perp_0}(\Acal^{\perp_0}),\Acal^{\perp_0})$; 
\item \textbf{cogenerated by $\Acal$} if $(\Ucal,\Vcal)=({}^{\perp_0}\Acal,({}^{\perp_0}\Acal)^{\perp_0})$;
\item \textbf{compactly generated} if $(\Ucal,\Vcal)$ is generated by a set of compact objects.
\end{itemize}
{Moreover, we say that $\Acal$  \textbf{generates} $\T$ if the subcategory $\bigcup_{n\in\mathbb{Z}}\Acal[n]$ generates the torsion pair $(\Tcal,0)$. Dually, we say that $\Acal$ \textbf{cogenerates} $\T$ if the subcategory $\bigcup_{n\in\mathbb{Z}}\Acal[n]$ cogenerates the torsion pair $(0,\Tcal)$.}
\end{definition}

Recall that a subcategory $\Ucal$ of $\T$ is said to be \textbf{suspended} (respectively, \textbf{cosuspended}) if it is closed under extensions and positive (respectively, negative) shifts. For example, a torsion pair $(\Ucal,\Vcal)$ is a t-structure if and only if $\Ucal$ is suspended (or, equivalently, $\Vcal$ is cosuspended). In particular, a t-structure generated (respectively, cogenerated) by a subcategory $\Acal$ is also generated (respectively, cogenerated) by the smallest suspended (respectively, cosuspended) subcategory containing $\Acal$. A dual statement holds for co-t-structures.

\begin{definition}
Two torsion pairs of the form $(\Ucal,\Vcal)$ and $(\Vcal,\Wcal)$ are said to be \textbf{adjacent}. More precisely, we say that $(\Ucal,\Vcal)$ is \textbf{left adjacent} to $(\Vcal,\Wcal)$ and that $(\Vcal,\Wcal)$ is \textbf{right adjacent} to $(\Ucal,\Vcal)$. Such $\Vcal$ is then called a \textbf{TTF (torsion-torsion-free) class} and the triple $(\Ucal,\Vcal,\Wcal)$ is said to be a \textbf{TTF triple}. Moreover, a TTF triple $(\Ucal,\Vcal,\Wcal)$ is said to be \textbf{suspended} (respectively, \textbf{cosuspended}) if the corresponding TTF class is a suspended (respectively, cosuspended) subcategory of $\T$.
\end{definition}

Note that, in a TTF triple, one of the torsion pairs is a t-structure if and only if the adjacent one is a co-t-structure. 

\begin{example}\label{ex TTF}
\begin{enumerate}
\item Let $A$ be a ring and consider its derived category $\mathsf{D}(A)$. Denote by $D^{\leq -1}$ (respectively, $D^{\geq 0}$) the subcategory of $\mathsf{D}(A)$ formed by the complexes whose usual complex cohomology vanishes in all non-negative degrees (respectively, in all negative degrees). The pair $(D^{\leq -1},D^{\geq 0})$ is a nondegenerate t-structure in $\mathsf{D}(A)$, called the \textbf{standard t-structure}. We note that the standard t-structure admits both a left and a right adjacent co-t-structure. We refer to \cite[Example 2.9(2)]{AMV1} for details on the left adjacent co-t-structure. Analogously, the right adjacent co-t-structure is the pair $(D^{\geq 0},K_{\leq -1})$ where $K_{\leq -1}$ stands for the subcategory of objects in $\mathsf{D}(A)$ which are isomorphic to a complex $X^\bullet$ of injective $A$-modules such that $X^i=0$ for all $i\geq 0$. The triple $(D^{\leq -1},D^{\geq 0},K_{\leq -1})$ is then a cosuspended TTF triple. Clearly, the heart of $(D^{
 \leq -1},D^{\geq 0})$ is $\mathsf{Mod}(A)$ and the coheart of $(D^{\geq 0},K_{\leq -1})$ coincides with $\mathsf{Inj}(A)$.
\item It follows from {\cite[Theorem 4.3]{AI}} that if $\Acal$ is a set of compact objects, then the pair $({}^{\perp_0}(\Acal^{\perp_0}),\Acal^{\perp_0})$ is a torsion pair. If $\T$ is moreover an algebraic triangulated category, then such a pair admits a right adjacent torsion pair, as shown in \cite[Theorem 3.11]{StPo}. In this case, if $\Acal$ is a suspended (respectively, cosuspended) subcategory of $\T^c$, then the triple $({}^{\perp_0}(\Acal^{\perp_0}),\Acal^{\perp_0},(\Acal^{\perp_0})^{\perp_0})$ is a cosuspended (respectively, suspended) TTF triple. We investigate some properties of the heart of compactly generated cosuspended TTF triples in Section 4.
\item Following the arguments in \cite[Proposition 1.4]{Nee}, we have that if $\Vcal$ is a cosuspended and preenveloping (respectively, suspended and precovering) subcategory of $\T$, then the inclusion of $\Vcal$ in $\T$ has a left (respectively, right) adjoint. In particular, there is a t-structure $(\Ucal,\Vcal)$ (respectively, a t-structure $(\Vcal,\Wcal)$) in $\T$. 
In our context, this shows that a co-t-structure $(\Vcal,\Wcal)$ has a left (respectively, right) adjacent t-structure if and only if $\Vcal$ is preenveloping (respectively, $\Wcal$ is precovering).
\end{enumerate}
\end{example}


\subsection{(Co)Silting}\label{(Co)silting} Recall the definition of silting and cosilting objects in a triangulated category (see \cite{PV}).
\begin{definition}
An object $M$ in $\T$ is called:
\begin{itemize}
\item \textbf{silting} if $(M^{\perp_{>0}},M^{\perp_{\leq 0}})$ is a t-structure in $\T$ and $M\in M^{\perp_{>0}}$;
\item \textbf{cosilting} if $({}^{\perp_{\leq 0}}M,{}^{\perp_{>0}}M)$ is a t-structure in $\T$ and $M\in{}^{\perp_{>0}}M$.
\end{itemize}
We say that two silting (respectively, cosilting) objects are \textbf{equivalent}, if they give rise to the same t-structure in $\T$ and
we call such a t-structure \textbf{silting} (respectively, \textbf{cosilting}).
The heart of the t-structure associated to a silting or cosilting object $M$ is denoted by $\Hcal_M$ and the cohomological functor $\T\longrightarrow \Hcal_M$ by $H^0_M$.
\end{definition}

It follows from the definition that silting and cosilting t-structures are nondegenerate and that a silting (respectively, cosilting) object generates (respectively, cogenerates) the triangulated category $\T$ (see \cite{PV}). 

\begin{example}\label{example (co)silting}
Let $A$ be a ring and $\mathsf{D}(A)$ its derived category.
\begin{enumerate}
\item Let $E$ be an injective cogenerator of $\mathsf{Mod}(A)$. Regarded as an object in $\mathsf{D}(A)$, $E$ is a cosilting object and the associated cosilting t-structure is the standard one. As discussed in Example \ref{ex TTF}(1), there is also a right adjacent co-t-structure with coheart $\Prod(E)=\mathsf{Inj}(A)$.
\item It follows from \cite[Theorem 4.6]{AMV1} that a silting object $T$ of $\mathsf{D}(A)$ lying in $\mathsf{K}^b(\mathsf{Proj}(A))$ gives rise to a suspended TTF triple, that is, the silting t-structure $(T^{\perp_{>0}},T^{\perp_{\leq 0}})$ admits a left adjacent co-t-structure with coheart $\Add(T)$ (see also \cite{Wei}). 
Dually, a cosilting object $C$ of $\mathsf{D}(A)$ lying in $\mathsf{K}^b(\mathsf{Inj}(A))$ gives rise to a cosuspended TTF triple, that is, the cosilting t-structure $({}^{\perp_{\leq 0}}C,{}^{\perp_{> 0}}C)$ admits a right adjacent co-t-structure with coheart $\Prod(C)$. 
For this dual statement, we refer to forthcoming work in \cite{MV2}.
\end{enumerate}
\end{example}

Silting and cosilting objects produce hearts with particularly interesting homological properties. First, recall from \cite{PaSa} that hearts of t-structures in a triangulated category with products and coproducts also have products and coproducts. Indeed, the (co)product of a family of objects in the heart is obtained by applying the functor $H^0_\mathbb{T}$ to the corresponding (co)product of the same family in $\T$. Of course, this (co)product in the heart may differ from the (co)product formed in $\T$.
\begin{lemma}\cite[Proposition 4.3]{PV}\label{properties heart} Let $M$ be a silting (respectively, cosilting) object in $\T$. Then the heart $\Hcal_M$ is an abelian category with a projective generator (respectively, an injective cogenerator) given by $H_M^0(M)$.
\end{lemma}

The following lemma establishes a particularly nice behaviour of the cohomological functors arising from (co)silting t-structures with respect to products and coproducts.

\begin{lemma}\label{Add and Prod} 
If $T$ is a silting object in $\T$, then the functor $H^0_T$ induces an equivalence between $\Add_\Tcal(T)$ and $\Add_{\Hcal_T}(H^0_T(T))=\mathsf{Proj}(\Hcal_T)$. Dually, if $C$ is a cosilting object in $\T$, then the functor $H^0_C$ induces an equivalence between $\Prod_\Tcal(C)$ and $\Prod_{\Hcal_C}(H^0_C(C))=\mathsf{Inj}(\Hcal_C)$.
\end{lemma}
\begin{proof}
We prove the statement for a cosilting object $C$ in $\T$ (the silting case is shown dually). Let the truncation functors of the associated cosilting t-structure $({}^{\perp_{\leq 0}}C,{}^{\perp_{>0}}C)$ be denoted by $u:\Tcal\longrightarrow {}^{\perp_{\leq 0}}C$ and $v:\Tcal\longrightarrow {}^{\perp_{>0}}C$. Recall from \cite[Lemma 4.5(iii)]{PV} that ${\Prod_\Tcal}(C)={}^{\perp_{>0}}C\cap ({}^{\perp_{>0}}C[-1])^{\perp_0}$. We first show that $H^0_C$ is fully faithful on $\Prod_\Tcal(C)$ (compare with \cite[Lemma 5.1(d)]{KV}, \cite[Lemma 1.3]{AST} and \cite[Lemma 3.2]{NSZ}). Let $X_1$ and $X_2$ be objects in $\Prod_{\T}(C)$. Suppose that $f:X_1\longrightarrow X_2$ is a map in $\T$ such that {$H^0_C(f)=0$}. Since, by assumption, $X_i$ (i=1,2) lies in ${}^{\perp_{>0}}C$, there is a truncation triangle of the form
$$\xymatrix{v(X_i[1])[-2]\ar[r]^{\ \ \ \ \ \kappa_i}&H^0_C(X_i)\ar[r]^{\ \ \mu_i}&X_i\ar[r]& v(X_i[1])[-1]}.$$
Now $f$ induces a morphism of triangles and, in particular, we have that $0=\mu_2{H^0_C}(f)=f\mu_1$. Thus, $f$ factors through $v(X_1[1])[-1]$. However, since $X_2$ lies in $({}^{\perp_{>0}}C[-1])^{\perp_0}$, we have that $\Hom_\Tcal(v(X_1[1])[-1],X_2)=0$ and, therefore, $f=0$. Now let us show that ${H^0_C}$ is also full on ${\Prod_\Tcal}(C)$. Suppose that $g$ is a map in $\Hom_\Tcal(H^0_C(X_1),H^0_C(X_2))$. Since $X_2$ lies in $({}^{\perp_{>0}}C[-1])^{\perp_0}$, the composition $\mu_2g\kappa_1$ vanishes and, therefore, there is a map $\tilde{g}:X_1\longrightarrow X_2$ such that $\tilde{g}\mu_1=\mu_2g$. Therefore, $g$ extends to a morphism of triangles and, as a consequence, $g=H^0_C(\tilde{g})$.

It remains to show that the essential image of $H_C^0$ restricted to $\Prod_\T(C)$ coincides with $\Prod_{\Hcal_C}(H^0_C(C))$. Observe first that $H_C^0(\prod_{i\in I}X_i)=\prod_{i\in I}H^0_C(X_i)$ for every family $(X_i)_{i\in I}$ of objects in $\mathsf{Prod}_\Tcal(C)$, where the product of the family $(H^0_C(X_i))_{i\in I}$ is taken in $\Hcal_C$. The proof is dual to  the argument for silting objects in \cite[Lemma 3.2.2(a)]{NSZ}. Take now  an object $M$ in $\mathsf{Prod}_{\Hcal_C}(H^0_C(C))$ and let $N$ be an object in $\Hcal_C$ such that $M\oplus N=H^0_C(C)^{I}$ for some set $I$. Then there is an idempotent element $e_M$ in $\mathsf{End}_{\Hcal_C}(H^0_C(C)^{I})=\mathsf{End}_{\Hcal_C}(H^0_C(C^{I}))$ whose image is the summand $M$. Since $H^0_C$ is fully faithful on ${\mathsf{Prod}_\Tcal}(C)$, it follows that there is an idempotent element $e$ in $\mathsf{End}_{\T}(C^{I})$ such that $H_C^0(e)=e_M$. Given that $\T$ is idempotent complete, the map $e$ factors as $C^I\stackrel{f}{\to}X\stackrel{g}{\to} C^I$ such that $fg={\rm id}_X$, and it then follows that ${H^0_C}(X)=M$.
\end{proof}

We finish this section with a general observation on abelian categories that will be useful later.

\begin{lemma}\label{extend injectives}
Let $\Acal$ and $\Bcal$ be abelian categories with enough injective (respectively, projective) objects and let $F:\Acal\longrightarrow \Bcal$ be a left (respectively, right) exact functor yielding an equivalence $\mathsf{Inj}(\Acal)\longrightarrow \mathsf{Inj}(\Bcal)$ (respectively, $\mathsf{Proj}(\Acal)\longrightarrow \mathsf{Proj}(\Bcal)$). Then $F$ is an equivalence of abelian categories.
\end{lemma}
\begin{proof}
Suppose that $\Acal$ and $\Bcal$ have enough injective objects. Then both categories can be recovered as factor categories of the corresponding categories $\mathsf{Mor}(\mathsf{Inj}(\Acal))$ and $\mathsf{Mor}(\mathsf{Inj}(\Bcal))$ of morphisms between injectives. Indeed, the kernel functors induce equivalences $\Ker_\Acal:\mathsf{Mor}(\mathsf{Inj}(\Acal))/\Rcal_\Acal\longrightarrow \Acal$ and $\Ker_\Bcal:\mathsf{Mor}(\mathsf{Inj}(\Acal))/\Rcal_\Bcal\longrightarrow \Bcal$, where the relations $\Rcal_\Acal$ and $\Rcal_\Bcal$ are the obvious ones (compare with \cite[Proposition IV.1.2]{ARS} for the case of projectives). Since $F$ induces an equivalence between $\mathsf{Inj}(\Acal)$ and $\mathsf{Inj}(\Bcal)$, it clearly also induces an equivalence between the corresponding morphism categories and, moreover, since $F$ is left exact, it indeed defines an equivalence $\widetilde{F}:\mathsf{Mor}(\mathsf{Inj}(\Acal))/\Rcal_\Acal\longrightarrow \mathsf{Mor}(\mathsf{Inj}(\Bcal))/\Rcal_\Bcal$ such that $Ker_\Bcal\circ \widetilde{F}=F\circ \Ker_\Acal$. Hence, $F$ is an equivalence. The dual statement follows analogously.
\end{proof}


\section{Grothendieck hearts in compactly generated triangulated categories}
Recall that a Grothendieck category is an abelian category with coproducts, exact direct limits and a generator. It is well-known that Grothendieck categories have enough injective objects and every object admits an injective envelope.
This section is dedicated to the question of determining when hearts of silting and cosilting t-structures are Grothendieck categories. We answer this question using a suitable category of functors and a corresponding theory of purity. 
We begin this section with a quick reminder of the relevant concepts.
\subsection{Functors and purity}
We consider the category $\Modr\Tcal^c$ of contravariant additive functors from $\T^c$ to $\mathsf{Mod}(\mathbb{Z})$, which is known to be a locally coherent Grothendieck category (see \cite{Kr0}, \cite[Subsection 1.2]{Kr1}). 

Consider the {restricted Yoneda functor} 
$$\mathbf{y}:\Tcal\longrightarrow \Modr\Tcal^c, \ \ \ \ \ \mathbf{y}X=\Hom_\Tcal(-,X)_{|\Tcal^c},\ \forall X\in \Tcal.$$
It is well-known that $\mathbf{y}$ is not, in general, fully faithful. 
A triangle 
$$\Delta:\ \ \ \xymatrix{X\ar[r]^f&Y\ar[r]^g&Z\ar[r]&X[1]}$$ 
in $\T$ is said to be a \textbf{pure triangle} if $\mathbf{y}\Delta$ is a short exact sequence. In other words, the triangle $\Delta$ is pure, if for any compact object $K$ in $\T$, the sequence
$$\xymatrix{0\ar[r]&\Hom_{\T}(K,X)\ar[rr]^{\Hom_\T(K,f)}&&\Hom_\T(K,Y)\ar[rr]^{\Hom_\T(K,g)}&&\Hom_\T(K,Z)\ar[r]&0}$$
is exact. We say that a morphism $f:X\rightarrow Y$ in $\T$ is a \textbf{pure monomorphism} (respectively, a \textbf{pure epimorphism}) if $\mathbf{y}f$ is a monomorphism (respectively, an epimorphism) in $\Modr\T^c$. An object $E$ of $\T$ is said to be \textbf{pure-injective} if any pure monomorphism $f:E\rightarrow Y$ in $\Tcal$ splits. Similarly, an object $P$ is said to be \textbf{pure-projective} in $\T$ if any pure epimorphism $g:X\rightarrow P$ splits.

The following theorem collects useful properties of pure-injective and pure-projective objects.
\begin{theorem}\cite[Theorem 1.8, Corollary 1.9]{Kr1}\cite[Section 11]{Bel}\label{Krause}
The following statements are equivalent for an object $E$ in $\Tcal$.
\begin{enumerate}
\item $E$ is pure-injective;
\item $\mathbf{y}E$ is an injective object in $\Modr\T^c$;
\item The map $\Hom_\Tcal(X,E)\rightarrow \Hom_{\Modr\Tcal^c}(\mathbf{y}{X},\mathbf{y}{E})$, {$\phi\mapsto\mathbf{y}\phi$} is an isomorphism for any {object} $X$ in $\Tcal$;
\item For every set $I$, the summation map $E^{(I)}\rightarrow E$ factors through the canonical map $E^{(I)}\rightarrow E^I$.
\end{enumerate}
Dually, the following are equivalent for an object $P$ in $\Tcal$.
\begin{enumerate}
\item $P$ is pure-projective;
\item $\mathbf{y}P$ is a projective object in $\Modr\T^c$;
\item The map $\Hom_\Tcal(P,Y)\rightarrow \Hom_{\Modr\Tcal^c}(\mathbf{y}{P},\mathbf{y}{Y})$, {$\phi\mapsto\mathbf{y}\phi$} is an isomorphism for any {object} $Y$ in $\Tcal$;
\item $P$ lies in $\Add(\T^c)$.
\end{enumerate}
Moreover, any projective (respectively, injective) object in $\Modr\Tcal^c$ is of the form $\mathbf{y}P$ (respectively, $\mathbf{y}E$), for a pure-projective object $P$ (respectively, a pure-injective object $E$), uniquely determined up to isomorphism.
\end{theorem}

It follows from above that $\T$ has enough pure-injective objects and that every object $X$ in $\T$ admits a pure-injective envelope.
The following theorem collects two results that will become essential later on.

\begin{theorem}\label{lift functor}\label{BelKr}
Let $H:\T\rightarrow \Acal$ be a cohomological functor from $\T$ to an abelian category $\Acal$.
\begin{enumerate}
\item \cite[Theorem 3.4]{Bel} If $H$ sends pure triangles in $\T$ to short exact sequences in $\Acal$, then there is a unique exact functor $\overline{H}:\Modr\T^c\rightarrow\Acal$ such that $\overline{H}\circ \mathbf{y}=H$.
\item \cite[Corollary 2.5]{Kr1} If $\Acal$ has exact direct limits and $H$ preserves coproducts, then $H$ sends pure triangles in $\T$ to short exact sequences in $\Acal$.
\end{enumerate}
\end{theorem}
We recall from \cite{Bel} how to construct $\overline{H}$. Given $F$ in $\Modr\T^c$, consider an injective copresentation
$$\xymatrix{0\ar[r]&F\ar[r]&\mathbf{y}E_0\ar[r]^{\mathsf{y}\alpha}&\mathbf{y}E_1}$$
where $E_0$ and $E_1$ are pure-injective in $\T$ and $\alpha$ is a map in $\Hom_\T(E_0,E_1)$. Then we define $\overline{H}(F):=\Ker H(\alpha)$, and it can be checked that $\overline{H}$ is indeed well-defined (that is, it does not depend on the choice of the injective copresentation of $F$). This functor can also be obtained in a dual way by taking a projective presentation of $F$.


\subsection{Grothendieck hearts and purity}
Note that, in general, the cohomological functor associated to a t-structure does not commute with products and coproducts in $\T$. The following lemma provides necessary and sufficient conditions for this to happen.

\begin{lemma}\label{smash}
Let $\mathbb{T}=(\Ucal,\Vcal)$ be a nondegenerate t-structure in $\T$ with heart $\Hcal_\mathbb{T}$ and associated cohomological functor $H^0_\mathbb{T}:\T\rightarrow \Hcal_\mathbb{T}$. Then the functor $H^0_\mathbb{T}$ preserves $\Tcal$-coproducts (respectively, $\Tcal$-products) if and only if $\Vcal$ is closed under coproducts (respectively, $\Ucal$ is closed under products). 

If these conditions are satisfied, we say that $\mathbb{T}$ is \textbf{smashing} (respectively, \textbf{cosmashing}).
\end{lemma}

\begin{proof}
We prove the statement for coproducts; the statement for products follows dually.
{Notice that aisles are always closed under coproducts. If also the coaisle $\Vcal$ is closed under coproducts, then both truncation functors} $u:\T\longrightarrow \Ucal$ and $v:\T\longrightarrow \Vcal$ commute with $\Tcal$-coproducts and, hence, so does $H^0_\mathbb{T}$. In particular, coproducts in $\Hcal_\mathbb{T}$ coincide with coproducts in $\T$.
For the converse, it is easy to check that nondegenerate t-structures can be cohomologically described, i.e. $\Vcal$ can be described as the subcategory formed by objects $X$ such that $H^0_\mathbb{T}(X[k])=0$ for all $k<0$. Consequently, since $H^0_\mathbb{T}$ commutes with $\T$-coproducts, this description shows that $\Vcal$ is closed under coproducts.
\end{proof}

\begin{example}\label{ex smash}
\begin{enumerate}
\item By definition, every silting t-structure is cosmashing and every cosilting t-structure is smashing.
\item If a silting object $T$ is pure-projective, then the associated t-structure is smashing. Indeed, let $(X_i)_{i\in I}$ be a family of objects in $T^{\perp_{<0}}$ and let $X$ be their coproduct in $\T$. Since $T$ is pure-projective, we have that $\Hom_\T(T,X[n])\cong \Hom_{\Modr\T^c}(\mathbf{y}T,\mathbf{y}X[n])$ for all $n$ in $\mathbb{Z}$. The statement then follows from the fact that $\mathbf{y}$ commutes with coproducts and $\Ker\Hom_{\Modr\T^c}(\mathbf{y}T,-)$ is coproduct-closed. 
\item If a cosilting object $C$ is pure-injective, in general, it does not follow that the associated t-structure is cosmashing. Indeed, let $A$ be the Kronecker algebra and let $C$ be the Reiten-Ringel cotilting module from \cite[Proposition 10.1]{RR} with associated torsion pair $(\Qcal\,,Cogen(C))$ in $\mathsf{Mod}(A)$, where $\Qcal$ is the class of all modules generated by preinjective $A$-modules. The object $C$ is cosilting in $D(A)$ (see \cite[Theorem 4.5]{St}). Note that, since $C$ is pure-injective in $\mathsf{Mod}(A)$ by \cite{B}, it follows from Theorem~\ref{Krause} that $C$ is also pure-injective when viewed as an object in $D(A)$. It turns out that the aisle of the associated cosilting t-structure consists precisely of those complexes whose zeroth cohomology belongs to $\Qcal$ and for which all positive cohomologies vanish (compare with \cite{HRS}). In particular, the cosilting t-structure is cosmashing if and only if $\Qcal$ is closed under products in $\mathsf{Mod}(A)$. But the latter cannot be true due to \cite[Theorem 5.2 and Example 5.4]{AH}.
\end{enumerate}
\end{example}

For a compactly generated triangulated category $\T$, (co)silting t-structures can be obtained in a rather abstract way. First, recall that $\T$ satisfies a \textbf{Brown representability theorem} (i.e. every cohomological functor $H:\T^{op}\rightarrow \mathsf{Mod}(\mathbb{Z})$ which sends coproducts to products is representable) and a \textbf{dual Brown representability theorem} (i.e. every cohomological functor $H:\T\rightarrow \mathsf{Mod}(\mathbb{Z})$ which sends products to products is representable) - see \cite{Kr3} for details. We can now state the following result.

\begin{theorem}\cite[Section 4]{NSZ}\label{t-structures and representability}
There is a bijection between
\begin{itemize} 
\item cosmashing nondegenerate t-structures whose heart has a projective generator;
\item equivalence classes of silting objects.
\end{itemize}
Dually, there is a bijection between
\begin{itemize}
\item smashing nondegenerate t-structures whose heart has an injective cogenerator; 
\item equivalence classes of cosilting objects.
\end{itemize}
\end{theorem}

The first statement is proven in \cite{NSZ}. For the reader's convenience, we briefly sketch an argument for the second bijection. First recall that cosilting t-structures are smashing, nondegenerate and their hearts have injective cogenerators (see Lemma \ref{properties heart}). Hence, there is an injective assignment from equivalence classes of cosilting objects to the t-structures with the assigned properties. To see that the assignment is surjective, we use the fact that $\T$ satisfies Brown representability. Indeed, given a smashing nondegenerate t-structure $\mathbb{T}$ whose heart has injective cogenerator $E$, the corresponding cosilting object $C$ can be obtained as the (unique) representative of the cohomological functor $\Hom_\T(H^0_\mathbb{T}(-),E)\cong\Hom_\T(-,C)$. Note that $Hom_\T(H^0_\mathbb{T}(-),E)$ sends coproducts to products by the smashing assumption. The dual arguments were used in \cite{NSZ} to show the silting case.

We are now able to prove the main result of this section by building on Theorem \ref{t-structures and representability} and identifying which (co)silting t-structures have Grothendieck hearts. A similar result was obtained independently in \cite[Proposition 4.2]{NSZ} with the additional assumption that all t-structures considered are cosmashing.

\begin{theorem}\label{pure-inj Grothendieck}
Let $\mathbb{T}=(\Ucal,\Vcal)$ be a smashing nondegenerate t-structure in $\T$ with heart $\Hcal_\mathbb{T}$. Denote by $H^0_\mathbb{T}:\T\rightarrow \Hcal_\mathbb{T}$ the associated cohomological functor. The following statements are equivalent.
\begin{enumerate}
\item $\Hcal_\mathbb{T}$ is a Grothendieck category;
\item There is a pure-injective cosilting object $C$ in $\T$ such that $\mathbb{T}=({}^{\perp_{\leq 0}}C,{}^{\perp_{>0}}C)$.
\end{enumerate}
If the above conditions are satisfied, there is a (unique) exact functor $\overline{H^0_\mathbb{T}}:\Modr\T^c\longrightarrow \Hcal_\mathbb{T}$ with a right adjoint $j_*$ such that $\overline{H^0_\mathbb{T}}\circ\mathbf{y}=H^0_\mathbb{T}$ and $j_*H^0_\mathbb{T}(C)\cong \mathbf{y}C$. Moreover, there is a localisation sequence of the form
$$\xymatrix@C=0.5cm{
\Ker\overline{H^0_\mathbb{T}}= {}^{\perp_0}\mathbf{y}C \ar[rrr]^{i_*} &&& \Modr\T^c \ar[rrr]^{\overline{H^0_\mathbb{T}}}   \ar @/^1.5pc/[lll]_{i^!} &&& \Hcal_\mathbb{T} \ar @/^1.5pc/[lll]_{j_*}
 } 
$$
\end{theorem}

\begin{proof}
Suppose that $\Hcal_\mathbb{T}$ is a Grothendieck category. By Theorem \ref{t-structures and representability}, $\mathbb{T}$ is a cosilting t-structure for a cosilting object $C$, such that $\Hom_{\T}(H^0_\mathbb{T}(-),E)\cong \Hom_\Tcal(-,C)$ for some injective cogenerator $E$ in $\mathcal{H}_\mathbb{T}$. It remains to show that $C$ is pure-injective. Since, by Lemma \ref{smash}, $H^0_\mathbb{T}$ commutes with $\T$-coproducts, Theorem \ref{BelKr}(2) shows that $H^0_\mathbb{T}$ sends pure triangles to short exact sequences. In particular, $\Hom_\Tcal(-,C)$ sends pure triangles to short exact sequences, showing that $C$ is indeed pure-injective.

Conversely, let $C$ be a pure-injective cosilting object in $\Tcal$ with associated t-structure $\mathbb{T}=({}^{\perp_{\leq 0}}C,{}^{\perp_{>0}}C)$. It follows that the functor $\Hom_\Tcal(-,C)$ is naturally equivalent to $\Hom_{\T}(H^0_\mathbb{T}(-),H^0_\mathbb{T}(C))$ and, therefore, also the functor $H^0_\mathbb{T}$ sends pure triangles to short exact sequences. Consequently, by Theorem \ref{BelKr}(1), there is a (unique) exact functor $\overline{H^0_\mathbb{T}}:\Modr\T^c\longrightarrow \Hcal_\mathbb{T}$ such that $\overline{H^0_\mathbb{T}}\circ\mathbf{y}=H^0_\mathbb{T}$. The following argument is inspired by the proof of \cite[Theorem 6.2]{St}. 
Consider the hereditary torsion pair in $\Modr\Tcal^c$ cogenerated by the injective object $\mathbf{y}{C}$, i.e. the pair $({}^{\perp_0}\mathbf{y}C,\mathsf{Cogen}(\mathbf{y}{C}))$. 
The quotient category $\mathcal{G}_C:=\Modr\Tcal^c/{}^{\perp_0}\mathbf{y}C$ is a Grothendieck category (see \cite[Proposition III.9]{Gabriel}) and the quotient functor $\pi:\Modr\Tcal^c\longrightarrow \Gcal_C$ admits a fully faithful right adjoint functor $\rho:\mathcal{G}_C\longrightarrow \Modr\Tcal^c$, with essential image $$\mathsf{Cogen}(\mathbf{y}C)\cap \Ker \Ext^1_{\Modr\Tcal^c}({}^{\perp_0}\mathbf{y}C,-)$$ (see \cite[Corollary of Proposition III.3]{Gabriel} and \cite[Section 11.1.1]{Prest}). 
In particular, as in the proof of \cite[Theorem 6.2]{St}, it follows that an object $X$ of $\mathcal{G}_C$ is injective if and only if $\rho(X)$ lies in $\mathsf{Prod}(\mathbf{y}C)$, i.e. the full subcategory of injective objects in $\mathcal{G}_C$ is equivalent to $\mathsf{Prod}(\mathbf{y}C)$ which, by Theorem \ref{Krause}, is further equivalent to $\mathsf{Prod}(C)$. Thus, using Lemma \ref{Add and Prod}, we get the following commutative diagram of equivalences.
$$\xymatrix{\mathsf{Inj}(\Gcal_C)\ar[r]^\rho& \mathsf{Prod}(\mathbf{y}C)\ar[d]_{\overline{H^0_\mathbb{T}}}&\mathsf{Prod}(C)\ar[dl]^{H^0_\mathbb{T}}\ar[l]_{\mathbf{y}}\\ & \mathsf{Prod}(H^0_\mathbb{T}(C))}$$
Hence, the functor $\overline{H^0_\mathbb{T}}\circ\rho$ yields an equivalence between the category of injective objects in $\mathcal{G}_C$ and the category of injective objects in $\Hcal_\mathbb{T}$. Since the functor $\overline{H^0_\mathbb{T}}\circ\rho$ is clearly left exact, by Lemma \ref{extend injectives}, it extends to an equivalence of categories  $\mathcal{G}_C\cong\Hcal_\mathbb{T}$ showing, in particular, that $\Hcal_\mathbb{T}$ is a Grothendieck category.

Assume now that $\mathbb{T}$ satisfies (1) and (2). We first show that $\Ker\overline{H^0_\mathbb{T}}={}^{\perp_0}\mathbf{y}C$. Indeed, if $F$ is an object in ${}^{\perp_0}\mathbf{y}C$, then $\Hom_{\Modr\Tcal^c}(\mathbf{y}\alpha,\mathbf{y}C)$ is an epimorphism for any map $\mathbf{y}\alpha:\mathbf{y}E_0\longrightarrow \mathbf{y}E_1$ between injective objects in $\Modr\T^c$ with $\Ker(\mathbf{y}\alpha)=F$. Using the pure-injectivity of $C$, we get that $\Hom_\T(\alpha,C)$ and, hence, also $\Hom_{\Hcal_\mathbb{T}}(H^0_\mathbb{T}(\alpha),H^0_\mathbb{T}(C))$ is an epimorphism. Since $H^0_\mathbb{T}(C)$ is an injective cogenerator of $\Hcal_\mathbb{T}$, it follows that $H^0_\mathbb{T}(\alpha)$ is a monomorphism and, thus, $\overline{H^0_\mathbb{T}}(F)=0$, by the construction of $\overline{H^0_\mathbb{T}}$. Finally, since this argument is reversible the desired equality holds.

Now, in order to show the existence of the localisation sequence above, it is enough to prove that the functor $\overline{H^0_\mathbb{T}}$ admits a right adjoint. To this end, since $\overline{H^0_\mathbb{T}}\circ\rho$ is an equivalence and $\pi$ has a right adjoint, it suffices to check that $\overline{H^0_\mathbb{T}}\cong\overline{H^0_\mathbb{T}}\circ\rho\circ\pi$. By using the unit of the adjunction $(\pi,\rho)$, we get a natural transformation of functors $\overline{H^0_\mathbb{T}}\longrightarrow\overline{H^0_\mathbb{T}}\circ\rho\circ\pi$. We need to see that it induces an isomorphism on objects. But this follows from the fact that kernel and cokernel of the natural map $X\longrightarrow \rho\pi(X)$, for $X$ in $\Modr\Tcal^c$, are torsion, that is, they belong to ${}^{\perp_0}\mathbf{y}C=\Ker\overline{H^0_\mathbb{T}}$. Finally, by using the adjunction $(\overline{H^0_\mathbb{T}},j_*)$, we get $j_* H^0_\mathbb{T}(C)\cong j_*\overline{H^0_\mathbb{T}}(\mathbf{y}C)\cong\mathbf{y}C$, finishing the proof.
\end{proof}

One can state a somewhat dual result for silting objects. 

\begin{theorem}
Let $\mathbb{T}=(\mathbb{T}^{\leq 0},\mathbb{T}^{\geq 0})$ be a smashing and cosmashing nondegenerate t-structure in $\T$ with heart $\Hcal_\mathbb{T}$. Denote by $H^0_\mathbb{T}:\T\rightarrow \Hcal_\mathbb{T}$ the associated cohomological functor. The following are equivalent.
\begin{enumerate}
\item $\Hcal_\mathbb{T}$ is a Grothendieck category with a projective generator;
\item There is a pure-projective silting object $T$ in $\T$ such that $\mathbb{T}=({T}^{\perp_{>0}},{T}^{\perp_{\leq 0}})$.
\end{enumerate}
If the above conditions are satisfied, there is a (unique) exact functor $\overline{H^0_\mathbb{T}}:\Modr\T^c\longrightarrow \Hcal_\mathbb{T}$ with a left adjoint $j_!$ such that $\overline{H^0_\mathbb{T}}\circ\mathbf{y}=H^0_\mathbb{T}$ and $j_!H^0_\mathbb{T}(T)\cong \mathbf{y}T$. Moreover, there is a recollement of the form
$$\xymatrix@C=0.5cm{
\Ker\overline{H^0_\mathbb{T}} \ar[rrr]^{i_*} &&& \Modr\T^c \ar[rrr]^{\overline{H^0_\mathbb{T}}}  \ar @/_1.5pc/[lll]_{i^*}  \ar @/^1.5pc/[lll]_{i^!} &&& \Hcal_\mathbb{T}. \ar @/_1.5pc/[lll]_{j_!} \ar @/^1.5pc/[lll]_{j_*}
 } 
 $$
\end{theorem}

\begin{proof}
The arguments are dual to those in the proof of Theorem \ref{pure-inj Grothendieck}.  
Note that the additional assumption of the t-structure being smashing comes into play through the use of Theorem \ref{lift functor}(2), which is needed in an essential way to prove the pure-projectivity of the associated silting object. On the other hand, we have seen in Example \ref{ex smash}(2) that the t-structure is smashing whenever $T$ is a pure-projective silting object. Finally, observe that we get a recollement rather than just a localisation sequence like in Theorem \ref{pure-inj Grothendieck}, since, in the given context, $\Ker\overline{H^0_\mathbb{T}}$ is closed under products and coproducts in $\Modr\T^c$ (see also \cite[Corollary 4.4]{PV0}).
\end{proof}

As an immediate consequence of these results, we can identify the t-structures with Grothendieck hearts within the bijections of Theorem \ref{t-structures and representability}.
\begin{corollary}
There is a bijection between 
\begin{itemize} 
\item  smashing nondegenerate t-structures of $\T$ whose heart is a Grothendieck category; 
\item equivalence classes of pure-injective cosilting objects. 
\end{itemize}
Dually, there is a bijection between 
\begin{itemize} 
\item smashing and cosmashing nondegenerate  t-structures in $\T$ whose heart is a Grothendieck category with a projective generator; 
\item equivalence classes of pure-projective silting objects. 
\end{itemize}
\end{corollary}


\section{Cosuspended TTF classes}
In this section, we focus on cosuspended TTF classes in a compactly generated triangulated category $\T$. We relate the properties of the previous section (namely, Grothendieck hearts and the pure-injectivity of the associated cosilting objects) with the definability of the cosuspended TTF class. As a consequence, if $\T$ is algebraic, nondegenerate compactly generated t-structures have Grothendieck hearts. 

\subsection{Coherent functors and definability}
We begin with a short reminder on coherent functors and definable subcategories of $\T$, and we obtain an easy (but useful) criterion to check whether a certain subcategory of $\T$ is definable or not. {We also prove that a definable subcategory $\Vcal$ of $\T$ is preenveloping, i.e. for any object $X$ in $\T$ there is a map $\phi:X\longrightarrow V$ with $V$ in $\Vcal$ such that $\Hom_\Tcal(\phi,V^\prime)$ is surjective for all $V^\prime$ in $\Vcal$.}

Recall from \cite[Proposition 5.1]{Kr2} that a covariant additive functor $F:\Tcal\longrightarrow \mathsf{Mod}(\mathbb{Z})$ is said to be \textbf{coherent} if the following equivalent conditions are satisfied
\begin{enumerate}
\item there are compact objects $K$ and $L$ and a presentation 
$$\Hom_\T(K,-)\longrightarrow \Hom_\T(L,-)\longrightarrow F\longrightarrow 0.$$
\item $F$ preserves products and coproducts and sends pure triangles to short exact sequences.
\end{enumerate}
The category of coherent functors is denoted by $\mathsf{Coh}$-$\T$. For a  locally coherent Grothendieck category {$\mathcal{G}$} or, more generally,  a locally finitely presented additive category {with products}, coherent functors are defined analogously, replacing  in (1) the compactness of $K$ and $L$  by the property of being finitely presented. The analogue of (2) then states that coherent functors are precisely the functors $\Gcal\longrightarrow \mathsf{Mod}(\mathbb{Z})$ preserving products and direct limits (\cite[Proposition 3.2]{Kr4}).

\begin{definition}
A subcategory $\Vcal$ of $\Tcal$ is said to be \textbf{definable} if there is a set of coherent functors $(F_i)_{i\in I}$ from $\T$ to $\mathsf{Mod}(\mathbb{Z})$ such that $X$ lies in $\Vcal$ if and only if $F_i(X)=0$ for all $i$ in $I$.
\end{definition}

Definable subcategories of a {locally finitely presented additive category $\Gcal$ with products} are defined as above: they are zero-sets of families of coherent functors. Recall that a subcategory of $\Gcal$ is definable if and only if it is closed under products, direct limits and pure subobjects (\cite[Theorem 2]{Kr4}). Moreover, definable subcategories of $\Gcal$ are closed under pure-injective envelopes (see \cite[Section 16.1.2]{Prest}).
Note that, by definition, also definable subcategories of $\T$ are closed under products, coproducts, pure subobjects and pure quotients, but we do not know whether they are characterised by such closure conditions (unless stronger assumptions are imposed, see \cite[Theorem 7.5]{Kr2}). A useful criterion for definability in $\Tcal$ will be provided in Corollary \ref{definable subcat} below.

\smallskip

For a subcategory $\Vcal$ of a compactly generated triangulated (respectively, a locally coherent Grothendieck) category, we denote by $\mathsf{Def}(\Vcal)$ the smallest definable subcategory containing $\Vcal$.

\begin{example}\label{flat}
A notion of flatness in $\Modr\Tcal^c$ is developed in \cite[Section 2.3]{Kr1} and \cite[Section 8.1]{Bel}. The subcategory $\mathsf{Flat}$-$\Tcal^c$ of flat objects in $\Modr\Tcal^c$ is locally finitely presented and contains precisely the functors $F$ that send triangles to exact sequences or, equivalently, that satisfy $\Ext^1(G,F)=0$ for all finitely presented functors $G$ in $\Modr\Tcal^c$. Moreover, $\mathsf{Flat}$-$\Tcal^c$ is a definable subcategory of $\Modr\Tcal^c$ by \cite[Theorem 16.1.12]{Prest}. Note that all objects of the form $\mathbf{y}X$, for $X$ in $\Tcal$, are flat.

The definable closure $\mathsf{Def}(\Vcal)$ in $\Modr\T^c$  of a set $\Vcal$ of objects  contained in $\mathsf{Flat}$-$\Tcal^c$ consists of pure subobjects of  direct limits in $\Modr\T^c$ of  directed systems whose terms are products of objects in $\Vcal$.  
Indeed,  to prove this, one uses  the notion of a reduced product from \cite[p.~465]{Kr5}. Since $\mathsf{Flat}$-$\Tcal^c$ is a definable subcategory of $\Modr\T^c$,  it suffices to show that the pure subobjects of reduced products of  objects in $\Vcal$ form a definable subcategory of $\mathsf{Flat}$-$\Tcal^c$. But the latter statement follows from \cite[Corollary 4.10]{Kr5} combined with  \cite[Proposition 2.2]{Kr5}.
\end{example}

We have the following useful fact (compare with \cite[Theorem 1.9]{ALPP}).

\begin{proposition}
Let $\mathsf{fun}(\mathsf{Flat}$-$\Tcal^c)$ denote the category of coherent functors {from the locally finitely presented category $\mathsf{Flat}$-$\Tcal^c$ to $\mathsf{Mod}(\mathbb{Z})$}. Then the assignment {$\mathsf{fun}(\mathsf{Flat}$-$\Tcal^c)\longrightarrow \mathsf{Coh}$-$\Tcal$} that sends a functor $F$ to $F\circ\mathbf{y}$ is an equivalence of categories.
\end{proposition}
\begin{proof}
First, we observe that the assignment is well-defined. It is clear that given $F$ in $\mathsf{fun}(\mathsf{Flat}$-$\Tcal^c)$, {the composition} $F\circ\mathbf{y}$ preserves products and coproducts. Now, given a pure triangle $\Delta$ in $\T$, we have that $\mathbf{y}(\Delta)$ is a short exact sequence in $\mathsf{Flat}$-$\T^c$. Since short exact sequences in $\mathsf{Flat}$-$\Tcal^c$ are pure or, equivalently, direct limits of split exact sequences (see \cite[Theorem 16.1.15]{Prest}), {we see that} $F(\mathbf{y}(\Delta))$ is a short exact sequence of abelian groups. It then follows that $F\circ\mathbf{y}$ is coherent by the description (2) of coherent functors above. 

In order to see that this assignment yields an equivalence of categories we show that it admits a quasi-inverse. By \cite[Proposition 4.1]{Kr2}, each functor $G$ in {$\mathsf{Coh}$-$\T$} gives rise to a unique functor $\overline{G}$ in $\mathsf{fun}(\mathsf{Flat}$-$\Tcal^c)$ such that $\overline{G}\circ\mathbf{y}=G$. The uniqueness guarantees the functoriality of this assignment and it is clear that the assignments are inverse to each other.
\end{proof}

 As a corollary of the proposition above, we deduce the following statement.

\begin{corollary}\label{definable subcat}\label{closure properties}
Let $\Vcal$ be a class of objects in $\T$. The smallest definable subcategory of $\T$ containing $\Vcal$ is $$\mathsf{Def}(\Vcal)=\{X\in\T:\mathbf{y}X\in\mathsf{Def}(\mathbf{y}\Vcal)\}.$$
As a consequence,  any definable subcategory of  $\Tcal$ is closed under pure-injective envelopes.
\end{corollary}

{Recall from \cite[Theorem 4.1]{CPT} that any definable subcategory of a locally finitely presented additive category $\Gcal$ with products is preenveloping. The following proposition establishes a triangulated analogue. Its proof is inspired by the proof of \cite[Theorem 4.3]{AI}.}

{\begin{proposition}\label{def is preenv}
Let $\Vcal$ be a definable subcategory of $\Tcal$. Then $\Vcal$ is preenveloping. In particular, if $\Vcal$ is cosuspended, then $({}^{\perp_0}\Vcal,\Vcal)$ is a t-structure.
\end{proposition}}

\begin{proof}
{Since $\Vcal$ is definable, by definition, there is a set of maps {$\{\phi_i:X_i\rightarrow Y_i| i \in I\}$} in $\Tcal^c$ such that an object $V$ in $\T$ lies in $\Vcal$ if and only if $\Hom_\Tcal(\phi_i,V)$ is surjective for all $i$ {in} $I$. We need to build a $\Vcal$-preenvelope for a given object $Z=Z_0$ in $\Tcal$. First, {setting $K_{i,0}:=\Hom_\Tcal(X_i,Z_0)$, we define the map}
$${\xymatrix{\overline{X}_0:=\bigoplus_{i\in I} X_i^{(K_{i,0})}\ar[rrr]^{\overline{\phi_0}:=\bigoplus_{i\in I}\phi_i^{(K_{i,0})}}&&&\overline{Y}_0:=\bigoplus_{i\in I} Y_i^{(K_{i,0})}}}$$
and consider the canonical universal map $a_0:\overline{X}_0\longrightarrow Z_0$. Let $z_0:Z_0\longrightarrow Z_1$ denote the corresponding component of the cone of the map $(\overline{\phi_0},-a_0)^T:\overline{X}_0\longrightarrow \overline{Y}_0\oplus Z_0$, and proceed inductively to define objects $Z_n$ and maps $z_n:Z_n\longrightarrow Z_{n+1}$. We prove that the Milnor colimit $V_Z$ of the inductive system $(Z_n,z_n)_{n\in\mathbb{N}_0}$ yields a $\Vcal$-approximation of $Z$. Let us first observe that indeed $V_Z$ lies in $\Vcal$. Since both $X_i$ and $Y_i$ are compact for any $i$ {in} $I$, it follows that 
$$\Hom_\Tcal(\phi_i,V_Z)\cong \varinjlim_{n\in\mathbb{N}_0}\Hom_\Tcal(\phi_i,Z_n):\varinjlim_{n\in\mathbb{N}_0}\Hom_\Tcal(Y_i,Z_n)\longrightarrow \varinjlim_{n\in\mathbb{N}_0}\Hom_\Tcal(X_i,Z_n).$$
In order to see that this map is surjective, it suffices to show that any element in $\varinjlim_{n\in\mathbb{N}_0}\Hom_\Tcal(X_i,Z_n)$ which is represented by a map $g$ in $\Hom_\Tcal(X_i,Z_{m})$ for some $m$ in $\mathbb{N}_0$ lies in the image of $\Hom_\Tcal(\phi_i,V_Z)$. By construction of the inductive system, there clearly is a map $h$ in $\Hom_\Tcal(Y_i,Z_{m+1})$ such that $h\phi_i=z_mg$. As a consequence, the element in $\varinjlim_{n\in\mathbb{N}_0}\Hom_\Tcal(Y_i,Z_n)$ represented by the map $h$ is a preimage via $\varinjlim_{n\in\mathbb{N}_0}\Hom_\Tcal(\phi_i,Z_n)$ of the element in $\varinjlim_{n\in\mathbb{N}_0}\Hom_\Tcal(X_i,Z_n)$ that we started with. This proves that $\Hom_\Tcal(\phi_i,V_Z)$ is surjective for all $i$ in $I$ and, thus, $V_Z$ lies in $\Vcal$.}

{We proceed to prove that the induced map $v:Z\longrightarrow V_Z$ is a left $\Vcal$-approximation. Given a morphism $f:Z\longrightarrow V$ with $V$ in $\Vcal$, the composition $fa_0$ factors through $\overline{\phi_0}$. Let $\overline{f}:\overline{Y}_0\longrightarrow V$ be such that ${\overline{f}\ \overline{\phi_0}}=fa_0$. By construction of $Z_1$ as the cone of $(\overline{\phi_0},-a_0)^T$, it follows that there is a map $f_1:Z_1\longrightarrow V$ such that $f=f_1z_0$. Inductively, one can then see that the map $f$ indeed factors successively through any $Z_n$ and, therefore, through the Milnor colimit $V_Z$, as wanted.}

{The final statement follows from Example \ref{ex TTF}(3).}
\end{proof}


\subsection{Cosuspended TTF triples}
Before we discuss how definability arises in the context of cosuspended TTF triples, we first prove {some} auxiliary statements.

\begin{lemma}\label{ttf is codet}
Let $(\Ucal,\Vcal,\Wcal)$ be a cosuspended TTF triple in $\T$. 
Then $(\Ucal,\Vcal)$ is a nondegenerate t-structure if and only if the coheart $\mathscr{C}:=\Vcal\cap\Wcal[-1]$ cogenerates $\T$. In this case, we have that $\Vcal={}^{\perp_{>0}}\mathscr{C}$.
\end{lemma}

\begin{proof}
Let $(\Ucal,\Vcal,\Wcal)$ be a cosuspended TTF triple in $\T$. Suppose that $(\Ucal,\Vcal)$ is a nondegenerate t-structure and let $X$ be an object of $\T$ such that $\Hom_\T(X,\mathscr{C}[k])=0$, for all $k$ in $\mathbb{Z}$. Given an integer $k$ in $\mathbb{Z}$, let us denote by $u^k:\Tcal\longrightarrow \Ucal[k]$ and $v^k:\Tcal\longrightarrow \Vcal[k]$ the truncation functors corresponding to the t-structure $(\Ucal[k],\Vcal[k])$. Consider a truncation triangle of the object $v^k(X)$ for the co-t-structure $(\Vcal[k-1],\Wcal[k-1])$, yielding a diagram of the form
$$\xymatrix{u^k(X)\ar[r]&X\ar[r]^{f^k}&v^k(X)\ar[r]&u^k(X)[1]\\ &V_{k-1}\ar[r]^g & v^k(X)\ar@{=}[u]\ar[r]^h& W_{k-1}\ar[r]& V_{k-1}[1]}$$
with $V_{k-1}$ in $\Vcal[k-1]$ and $W_{k-1}$ in $\Wcal[k-1]$. 
We can easily deduce that $W_{k-1}$ lies in $\mathscr{C}[k]$ and, thus, $hf^k=0$ by assumption on $X$. Then there is a morphism $\alpha:X\longrightarrow V_{k-1}$ such that $g\alpha=f^k$. Now, since $V_{k-1}$ lies in $\Vcal[k-1]\subseteq \Vcal[k]$, it follows that $\alpha$ factors through the truncation $f^{k-1}:X\longrightarrow v^{k-1}(X)$. This then yields a map $v^{k-1}(X)\longrightarrow v^k(X)$. Considering the two compositions of this map with the canonical morphism $v^k(X)\longrightarrow v^{k-1}(X)$ and using the minimality of the maps $f^k$ and $f^{k-1}$, we conclude that both maps are isomorphisms. Since this holds for arbitrary $k$, the nondegeneracy of $(\Ucal,\Vcal)$ implies that $v^k(X)=0$ for all $k$ in $\mathbb{Z}$. Thus, $X$ must lie in $\cap_{n\in\mathbb{Z}}\Ucal[n]$ and, again by the nondegeneracy of $(\Ucal,\Vcal)$ it must, therefore, be zero.

Conversely, suppose that $\mathscr{C}$ cogenerates $\T$ and let $X$ lie in $\cap_{n\in\mathbb{Z}}\Ucal[n]$. Consider a morphism $f:X\longrightarrow C[k]$ for $k$ in $\mathbb{Z}$ and $C$ in $\mathscr{C}$. Now, since $C[k]$ lies in $\Vcal[k]$ and $X$ lies in $\Ucal[k]$, it follows that $f=0$ and, thus, by assumption, also $X=0$. Similarly, if $X$ is in $\cap_{n\in\mathbb{Z}}\Vcal[n]$, since $C[k]$ lies in $\Wcal[k-1]$, it must follow that $X=0$. 

Finally, assuming that $\mathscr{C}$ cogenerates $\T$, we show that $\Vcal={}^{\perp_{>0}}\mathscr{C}$. It is always the case that $\Vcal\subseteq {}^{\perp_{>0}}\mathscr{C}$. For the reverse inclusion, let $X$ be an object in ${}^{\perp_{>0}}\mathscr{C}$ and consider the truncation triangle 
$$v(X)[-1]\longrightarrow u(X)\longrightarrow X\longrightarrow v(X).$$ 
Given $C$ in $\mathscr{C}$ and applying the functor $\Hom_\T(-,C[k])$ to the triangle, we see that $\Hom_\T(v(X),C[k+1])=\Hom_\T(X,C[k])=0$ for all $k>0$ and, thus, we have $\Hom_\T(u(X),C[k])=0$ for all $k>0$. Moreover, since $\mathscr{C}[k]\subset \Vcal$ for all $k\leq 0$, we see that $\Hom_\T(u(X),C[k])=0$ for all $k\leq 0$. Since $\mathscr{C}$ cogenerates $\Tcal$, we have that $u(X)=0$ and $X$ belongs to $\Vcal$, as wanted.
\end{proof}

\begin{lemma}\label{coheart is prod}
Let $\mathscr{C}$ be subcategory of $\Tcal$. Suppose that $\mathscr{C}$ is closed under products and summands, and that all objects in $\mathscr{C}$ are pure-injective. Then there is an object $C$ in $\mathscr{C}$ such that $\mathscr{C}=\mathsf{Prod}(C)$.
\end{lemma}
\begin{proof}
Consider the hereditary torsion pair $({}^{\perp_0}(\mathbf{y}\mathscr{C}),\mathcal{F}:=\mathsf{Cogen}(\mathbf{y}\mathscr{C}))$ in $\Modr\T^c$. It is well-known that there is an injective object $\mathbf{y}C$ in $\Modr\T^c$ such that $\mathcal{F}=\mathsf{Cogen}(\mathbf{y}C)$ (see \cite[VI, Proposition 3.7]{Sten}). 
It follows that $\Prod(\mathbf{y}\mathscr{C})=\Prod(\mathbf{y}C)$. Since $\mathbf{y}$ commutes with products and is fully faithful on pure-injectives, we get $\mathscr{C}=\Prod(\mathscr{C})=\Prod(C)$.
\end{proof}

\begin{lemma}\label{definable co-t}
Let $\mathscr{C}$ be an additive subcategory of $\Tcal$ and $\Vcal={}^{\perp_{>0}}\mathscr{C}$. Then the following statements are equivalent.
\begin{enumerate}
\item $\Vcal$ is product-closed and every object in $\mathscr{C}$ is pure-injective.
\item $\Vcal$ is definable.
\end{enumerate}
Moreover, if the above conditions are satisfied, then there is a t-structure $(\Ucal,\Vcal)$.
\end{lemma}
\begin{proof}
Suppose that (1) holds. 
We have to show that every  object $X$   in $\textsf{Def}(\Vcal)$ lies in $\Vcal={}^{\perp_{>0}}\mathscr{C}$. By Corollary \ref{definable subcat}, the object  $\mathbf{y}X$ lies in the definable closure  $\mathsf{Def}(\mathbf{y}\Vcal)$ in $\Modr\T^c$ of all objects of the form $\mathbf{y}V$ with $V$ in $\Vcal$.
By the description of the  definable closure given in  Example~\ref{flat}, $\mathbf{y}X$ is a pure subobject of a direct limit in $\Modr\T^c$ of a directed system whose objects are products of the form $\prod_{i\in I}\mathbf{y}X_i$, with $X_i$ in $\Vcal$. Note that, since $\mathbf{y}$ commutes with products, {we have} $\prod_{i\in I}\mathbf{y}X_i=\mathbf{y}X_I$, where $X_I=\prod_{i\in I}X_i$. Now, since $\Vcal$ is closed under products, $X_I$ lies in $\Vcal$. Applying the functor $\Hom_{\Modr\T^c}(-,\mathbf{y}C[k])$, with $k>0$ and $C$ in $\mathscr{C}$ to the embedding $\mathbf{y}X\longrightarrow \varinjlim_I \mathbf{y}X_I$, and using the pure-injectivity of $C$ (and its shifts), we get an epimorphism 
$$\varprojlim_I\Hom_\T(X_I,C[k])\cong \Hom_{\Modr\T^c}(\varinjlim_I \mathbf{y}X_I, \mathbf{y}C[k])\twoheadrightarrow \Hom_{\Modr\T^c}(\mathbf{y}X,\mathbf{y}C[k])\cong \Hom_\T(X,C[k]).$$
Since $X_I$ lies in $\Vcal$, the domain of this map vanishes and, hence, so does the codomain, as wanted. 

Conversely, suppose that $\Vcal$ is definable. First,  the subcategory $\Vcal$ is closed under products. Let $X$ be an object in $\mathscr{C}$ and $f:X\rightarrow I(X)$ its pure-injective envelope in $\T$. Since definable subcategories are closed under pure-injective envelopes and pure quotients (see  Corollary \ref{closure properties}), it follows that both $I(X)$ and $Z:=\mathsf{cone}(f)$ lie in $\Vcal$. Since $\Hom_\T(\Vcal,\mathscr{C}[1])=0$ it follows that $\Hom_\T(Z,X[1])=0$, the triangle induced by $f$ splits and $X$ is a summand of $I(X)$, i.e. $X$ is pure-injective.

{The last statement of the lemma follows from Proposition \ref{def is preenv}, since $\Vcal$ is clearly cosuspended.}
\end{proof}

{Finally, we can now use the rather technical statements above to prove the main theorem of this section.}

\begin{theorem}\label{triple is cosilting}
Let $(\Ucal,\Vcal,\Wcal)$ be a cosuspended TTF triple in $\T$ such that the t-structure $(\Ucal,\Vcal)$ is nondegenerate. Then the following are equivalent.
\begin{enumerate}
\item $\Vcal$ is definable in $\T$;
\item $\Vcal={}^{\perp_{>0}}C$ for a pure-injective cosilting object $C$ in $\T$.
\end{enumerate}
\end{theorem}
\begin{proof}
First observe that by Lemma \ref{ttf is codet}, the coheart $\mathscr{C}=\Vcal\cap\Wcal[-1]$ cogenerates $\T$ and $\Vcal={}^{\perp_{>0}}\mathscr{C}$. Since $\Vcal$ is a TTF class, it is closed under products and, therefore, by Proposition \ref{definable co-t}, $\Vcal$ is definable if and only if every object in $\mathscr{C}$ is pure-injective. In that case, {since both $\Vcal$ and $\Wcal$ (and, thus, $\mathscr{C}$) are closed under products and summands}, by Lemma \ref{coheart is prod}, there is $C$ in $\T$ such that $\mathscr{C}=\Prod(C)$. 

(1)$\Rightarrow$ (2): Suppose now that $\Vcal$ is definable and let $C$ be as above. As observed, we have that $\Vcal={}^{\perp_{>0}}C$ and we only need to show that $\Ucal={}^{\perp_{\leq 0}}C$. Since $C[k]$ lies in $\Vcal$ for all $k\leq 0$, it is clear that $\Ucal\subseteq {}^{\perp_{\leq 0}}C$. Now let $X$ lie in ${}^{\perp_{\leq 0}}C$ and consider a truncation triangle as follows
$$u(X)\longrightarrow X\longrightarrow v(X)\longrightarrow u(X)[1]$$
with $u(X)$ in $\Ucal$ and $v(X)$ in $\Vcal$. Since both $X$ and $u(X)[1]$ lie in ${}^{\perp_{\leq 0}}C$, so does $v(X)$. However, $v(X)$ also lies in ${}^{\perp_{>0}}C$, showing that $v(X)=0$, since $C$ is a cogenerator of $\T$. Hence, we have that $\Ucal={}^{\perp_{\leq 0}}C$.

(2)$\Rightarrow$ (1): In order to show that $\Vcal$ is definable, it is enough to show that the coheart $\mathscr{C}$ coincides with $\Prod(C)$ (thus proving that every object in $\mathscr{C}$ is pure-injective). The argument is dual to the one used in \cite[Lemma 4.5]{AMV1}. Indeed, let $X$ be an object in $\mathscr{C}$, let $I$ denote the set $\Hom_\Tcal(X,C)$ and consider the induced universal map $\phi:X\longrightarrow C^I$. If $Z$ denotes the cone of the map $\phi$, then it is easy to check that $Z$ lies in ${}^{\perp_{>0}}C$ and, thus, the map $Z\longrightarrow X[1]$ of the induced triangle is zero, by the assumption on $X$. Hence, the triangle splits and $X$ lies in $\Prod(C)$. This shows that $\mathscr{C}\subseteq \Prod(C)$ and the reverse inclusion is clear.
\end{proof}

\begin{corollary}\label{comp gen t-str}
Let $\T$ be an algebraic, compactly generated triangulated category. Every nondegenerate compactly generated t-structure has a Grothendieck heart.
\end{corollary}
\begin{proof}
From Example \ref{ex TTF}(2), every compactly generated t-structure $(\Ucal,\Vcal)$ in $\T$ admits a right adjacent co-t-structure $(\Vcal,\Wcal)$. It is clear that $\Vcal$ is definable as it is the subcategory obtained as the intersection of the kernels of the coherent functors $\Hom_\T(X,-)$ with $X$ in $\Ucal\cap\T^c$. Now, Theorem \ref{triple is cosilting} combined with Theorem \ref{pure-inj Grothendieck} finish the proof.
\end{proof}
The corollary above extends  \cite[Corollary 2.5]{BP}, which treats the special case when $\T$ is a  derived module category and the  compactly generated t-structure arises as an HRS-tilt of a torsion pair in the underlying module category.


\end{document}